\documentclass[12pt]{amsart}
\usepackage{amsfonts}
\usepackage{mathrsfs}
\usepackage{amsmath}
\usepackage{amsmath, amsthm, amssymb}
\usepackage{color}

\newcommand{\T}{\theta}

\renewcommand{\theequation}{\theequation. \arabic{equation}}
\numberwithin{equation}{section}
\newtheorem{thm}{Theorem}[section]

\newtheorem{lem}[thm]{Lemma}
\newtheorem{rem}[thm]{Remark}
\newtheorem{prop}[thm]{Proposition}
\newtheorem{defn}[thm]{Definition}

\def\squarebox#1{\hbox to #1{\hfill\vbox to #1{\vfill}}}

\begin{document}
\title[An extension of Ramanujan's reciprocity theorem]
{Extensions of Ramanujan's reciprocity theorem and the Andrews--Askey integral}
\author{Zhi-Guo Liu}
\footnote{Dedicated to the memory of my parents}\\
\footnote{This work was supported by the National Science Foundation of China
(Grant No. 11571114)  and Shanghai Key Laboratory of PMMP}
\date{\today}
\address{East China Normal University, Department of Mathematics, 500 Dongchuan Road,
Shanghai 200241, P. R. China} \email{zgliu@math.ecnu.edu.cn;
liuzg@hotmail.com}
\begin{abstract}
Ramanujan's reciprocity theorem may be considered as a three-variable extension of Jacobi's triple product identity.
Using the method of $q$-partial differential equations,  we extend Ramanujan's reciprocity theorem to a seven-variable reciprocity formula.  The Andrews--Askey integral is a $q$-integral having four parameters with base $q$. Using the same method we extend the Andrews--Askey integral
formula to a $q$-integral formula which has seven parameters with base $q$.
\end{abstract}
\maketitle
\section{Introduction}
\noindent In this paper we assume, unless otherwise stated,  that $|q|<1$ and use the standard product notation
\[
(a;q)_0 = 1, \quad (a;q)_n = \prod_{k=0}^{n-1}(1-aq^k)\quad
\text{and}\quad
(a;q)_\infty =\prod_{k=0}^\infty (1-aq^k).
\]
If $n$ is an integer or $\infty$,  the multiple $q$-shifted factorials are defined as
\begin{equation*}
(a_1, a_2,...,a_m; q)_n=(a_1;q)_n(a_2;q)_n \ldots (a_m;q)_n.
\end{equation*}

The celebrated Jacobi triple product identity is stated in the following proposition
(see, for example \cite[p. 1]{ChanHC} and \cite[p. 15]{Gas+Rah}).
\begin{prop}\label{jacobitri} For $x\not=0$, we have the triple product identity
\[
(q, x, q/x; q)_\infty=\sum_{n=-\infty}^\infty (-1)^n q^{n(n-1)/2} x^n.
\]
\end{prop}
This identity is among the most important identity in mathematics, which
has  many interesting applications in number theory, combinatorics,
analysis, algebra and mathematical physics. Some amazing extensions of this identity
have been made by various authors. Ramanujan's $_1\psi_1$ summation formula and
Bailey $_6\psi_6$ summation formula both contain this identity as a
special case, and  may be considered as two important extensions of this identity,
and these two extensions have wider applications than  Jacobi's triple product identity.

Ramanujan's reciprocity theorem and the Andrews--Askey integral formula may also be regarded as two
extensions of Jacobi's triple product identity. 

In this paper we will use the method of
$q$-partial differential equations to extend
Ramanujan's reciprocity theorem and Andrews--Askey integral formula to two more general $q$-formulae.

For simplicity, in this paper we use $\Delta(u, v)$ to denote the theta function
\begin{equation}
v(q, u/v, qv/u; q)_\infty.
\label{theta:eqn1}
\end{equation}
The $q$-binomial coefficients are the $q$-analogs of the binomial coefficients, which are defined by
\[
{n \choose k}_{q}=\frac{(q;q)_n}{(q;q)_k(q;q)_{n-k}}.
\]

As usual, the basic hypergeometric series or $q$-hypergeometric series
${_r\phi_s}$ is defined by
\begin{equation*}
{_r\phi_s} \left({{a_1, a_2, ..., a_{r}} \atop {b_1, b_2, ...,
b_r}} ;  q, z  \right) =\sum_{n=0}^\infty \frac{(a_1, a_2, ...,
a_{r};q)_n} {(q,  b_1, b_2, ..., b_s ;q)_n}\left((-1)^n q^{n(n-1)/2}\right)^{1+s-r} z^n.
\end{equation*}

Now we introduce the definition of the Thomae--Jackson $q$-integral in $q$-calculus,
which was introduced by Thomae \cite{Thomae} and Jackson \cite{Jackson1910}.
\begin{defn}
Given a function $f(x)$, the Thomae--Jackson $q$-integral of $f(x)$ on $[a, b]$ is defined by
\[
\int_{a}^b f(x)d_q x=(1-q)\sum_{n=0}^\infty [bf(bq^n)-af(aq^n)]q^n.
\]
\end{defn}
If the function $f(x)$ is continuous on $[a, b]$,  then, one can  deduce that
\[
\lim_{q\to 1}\int_{a}^b f(x)d_q x=\int_{a}^b f(x) dx.
\]

In $1981,$ Andrews and Askey \cite{AndrewsAskey} established the following interesting $q$-beta integral
formula using Ramanujan $_1\psi_1$ summation, which  has four parameters $a, b, u. v$
with base $q$, which is now known as the Andrews--Askey integral.
\begin{prop}\label{thmaa}
If  $\max\{|au|, |bu|, |av|, |bv|\}<1$ and $uv\not=0$, then, we have
\[
\int_{u}^v \frac{(qx/u, qx/v; q)_\infty}{(ax, bx; q)_\infty}d_q x
=\frac{(1-q)\Delta(u, v)(abuv; q)_\infty}{(au, bu, av, bv; q)_\infty}.
\]
\end{prop}
Subsequently, in $1982$,   Al--Salam and Verma \cite{SalamVerma} found that
Sears' nonterminating extension of the $q$-Saalsch\"{u}tz summation can be
rewritten in the following simple form, see also \cite[page~52 ]{Gas+Rah}.
\begin{prop}\label{salverpp} If  $\max\{|au|, |bu|, |cu|, |av|, |bv|, |cv|\}<1$ and $uv\not=0$,  then, we have
 \[
 \int_{u}^v \frac{(qx/u, qx/v, abcuvx; q)_\infty}{(ax, bx, cx; q)_\infty}d_q x
 =\frac{(1-q)\Delta(u, v)( abuv, acuv, bcuv; q)_\infty}{(au, bu, cu, av, bv, cv; q)_\infty}.
 \]
\end{prop}
We call this $q$-integral formula the Al--Salam--Verma integral formula.
This $q$-integral formula has five parameters with base $q$.
When $c=0,$ this $q$-integral formula reduces to
the Andrews--Askey integral in Proposition~\ref{thmaa}.

We have extended the Andrews--Askey integral formula or the  Al--Salam--Verma integral formula
to the following $q$-integral formula \cite[Proposition~13.8]{LiuRam2013},
which  has six parameters with base $q$.
Extensions of the Andrews--Askey integral involving the terminating $q$-series
have been discussed by \cite{Wang2008a} and \cite{Cao}.
\begin{prop}\label{liuqint} If $a, b, c, d, u, v, r$ are complex numbers such that \\
$\max\{|au|, |bu|, |cu|, |av|, |bv|, |cv|, |abr/c|\}<1$  and $uv\not=0$, then,  we have
the following $q$-integral formula:
\begin{align*}
\int_{u}^v \frac{(qx/u, qx/v, abrx; q)_\infty}{(ax, bx, cx;q)_\infty} d_q x
&=\frac{(1-q)\Delta(u, v) (acuv, bcuv, abr/c; q)_\infty} {(au, av, bu, bv, cu, cv; q)_\infty}\\
&\quad \times {_3\phi_2}\left({{cu, cv, cuv/r}\atop{acuv, bcuv}}; q, \frac{abr}{c}\right).
\end{align*}
\end{prop}
One of the main results of this paper is to extend the Andrews--Askey integral formula or
the  Al--Salam--Verma integral formula to the following integral formula, which
has seven parameters $a, b, c, d, r, u, v$ with bases $q$.
\begin{thm}\label{liuqtrans}  If $a, b, c, d, u, v, r$ are complex numbers such that $uv\not=0$
and
$\max\{|au|, |bu|, |cu|, |av|, |bv|, |cv|, |abr/c|\}<1,$ then, we have
\begin{align*}
&\int_{u}^v \frac{(qx/u, qx/v, acduvx, abrx; q)_\infty}{(ax, bx, cx, dx; q)_\infty}\\
&\qquad \times{_3\phi_2}\left({{ar, ax, cx}\atop{acduvx, abrx}}; q, bduv\right)d_q x\\
&=\frac{(1-q)\Delta(u, v)(acuv, aduv, bcuv, cduv, abr/c; q)_\infty}
{(au, av, bu, bv, cu, cv, du, dv; q)_\infty}\\
&\qquad \times {_3\phi_2}\left({{cu, cv, cuv/r}\atop{acuv, bcuv}}; q, \frac{abr}{c}\right).
\end{align*}
\end{thm}
When $d=0,$ the $_3\phi_2$ series in the integrand reduces to $1,$ and
in the same time  Theorem~\ref{liuqtrans} becomes Proposition~\ref{liuqint}.
So Theorem~\ref{liuqtrans} is really an extension of Proposition~\ref{liuqint}.

Setting $r=cuv$ in Theorem~\ref{liuqtrans}, the $_3\phi_2$ series on the right-hand side of the equation
in Theorem~\ref{liuqtrans} equals $1$, and we obtain the following proposition.
\begin{prop}\label{csalverliu}
If $a, b, c, d, u, v, r$ are complex numbers such that \\
$\max\{|au|, |bu|, |cu|, |av|, |bv|, |cv|\}<1$ and $uv\not=0$,  then, we have
\begin{align*}
&\int_{u}^v \frac{(qx/u, qx/v, abcuvx, acduvx; q)_\infty}{(ax, bx, cx, dx; q)_\infty}\\
&\qquad \times {_3\phi_2}\left({{acuv, ax, cx}\atop{abcuvx, acduvx}}; q, bduv\right)d_q x\\
&=\frac{(1-q)\Delta(u, v) (abuv, acuv, aduv, bcuv, cduv; q)_\infty}
{(au, av, bu, bv, cu, cv, du, dv; q)_\infty}.
\end{align*}
\end{prop}
In his lost notebook \cite[p 40]{Ramlost}, Ramanujan stated the following beautiful reciprocity
theorem without proof. This formula may be considered as a three-variable extension of Jacobi's triple product identity.
This result, now known as Ramanujan's reciprocity theorem, was first proved by Andrews
in his paper \cite{Andrews1981} in $1981$. For another proof, see \cite{B+C+Y+Y}.
\begin{thm}If $uv\not=0$ and $av\not=q^{-m}, au\not=q^{-m}$, m=0, 1, 2,\ldots,   then,  we have
 \begin{align*}\label{eq:Ram-Reci}
 \frac{(q, v/u, u/v; q)_\infty} {(au, av; q)_\infty}
 =(1-v/u)\sum_{n=0}^\infty (-1)^n \frac{q^{n(n+1)/2}(v/u)^n}{(av;q)_{n+1}}\\
 \hspace{4 cm}+(1-u/v)\sum_{n=0}^\infty (-1)^n \frac{q^{n(n+1)/2}(u/v)^n}{(au; q)_{n+1}}.
 \end{align*}
 \end{thm}
 Andrews \cite[Theorem~1]{Andrews1981} also derived a four-variable reciprocity theorem
 by using many summation and transformation formulae
for basic hypergeometric series.  Inspired by the work of Andrews, in $2003$, we \cite[Theorem~6]{Liu2003} proved
the following five-variable reciprocity formula by using the $q$-exponential operator to Ramanujan's $_1\psi_1$ summation.
\begin{thm}\label{liuppreciprocity}  For $\max\{|au|, |av|, |cu|, |cv|, |du|, |dv|\}<1$
and $uv\not=0$, then,  we have
\begin{align*}
&v\sum_{n=0}^\infty \frac{(q/du, acuv; q)_n(dv)^n}{(av, cv; q)_{n+1}}
-u\sum_{n=0}^\infty \frac{(q/dv, acuv; q)_{n}(du)^n}{(au, cu; q)_{n+1}}\\
&=\frac{\Delta(u, v)( aduv, acuv, cduv; q)_\infty}{(au, av, cu, cv, du, dv; q)_\infty}.
\end{align*}
\end{thm}
Ramanujan's reciprocity formula  is the special case $c=d=0$ of Theorem~\ref{liuppreciprocity},
and Ramanujan $_1\psi_1$ summation formula is the special case  $c=0$ of Theorem~\ref{liuppreciprocity}.
Setting $a=c=d=0$ in Theorem~\ref{liuppreciprocity} we can  immediately obtain the Jacobi triple
product identity.

On taking $cduv=q$ and $a=0,$ we can obtain the following very interesting Lambert series identity.
\begin{prop} \label{liulambert} If $cu\not=q^{-m}$ and $cv\not=q^{-m}$, m=0, 1, 2, \ldots,  then,  we have the Lambert series
identity
\[
v\sum_{n=0}^\infty \frac{(q/cu)^n}{1-cvq^n}
-u\sum_{n=0}^\infty \frac{(q/cv)^n}{1-cuq^n}
=\frac{v(q, q, u/v, qv/u; q)_\infty}{(cu, cv, q/cu, q/cv; q)_\infty}.
\]
\end{prop}

Recently, some other generalizations of Ramanujan's reciprocity formula  have
been found by various authors by rearranging some of the well-known $q$-formulae,
see, for example  \cite{ChuZhang, Kang, Ma}.  In this paper we will give a completely
new extension of Ramanujan's reciprocity formula.

For simplicity, we now introduce the notation $\rho$ in the following definition.
\begin{defn} \label{liudefn}We use the notation $\rho (a, b, c, d, r, u, v)$ to denote the double $q$-series
\begin{align*}
&v\sum_{n=0}^\infty \frac{(q/du, acuv, bcuv; q)_n (dv)^n}{(av, bv, cv; q)_{n+1}}\\
&\quad\times {_3\phi_2}\left({{q^{n+1}, vq^{n+1}/r, q/cu}\atop{avq^{n+1}, bvq^{n+1}}}; q, \frac{abcruv}{q}\right).
\end{align*}
\end{defn}
Our generalization of Ramanujan's reciprocity formula is the following reciprocity formula which
has seven parameters with base $q$.
\begin{thm}\label{ramliureci} If $\rho$ is defined as in Definition~\ref{liudefn} with
$uv\not=0$ and
\[
\max\{|au|, |av|, |bu|, |bv|, |cu|, |cv|, |du|, |dv|, |abr/d|, |abcruv/q|\}<1,
\]
then, we have the following  seven-variable reciprocity formula:
\begin{align*}
&\rho (a, b, c, d, r, u, v)-\rho (a, b, c, d, r, v, u)\\
&=\frac{\Delta(u, v)(acuv, aduv, bcuv, bduv, cduv, abr/d; q)_\infty}
{(au, av, bu, bv, cu, cv, du, dv, abcruv/q; q)_\infty}\\
&\quad\times {_3\phi_2}\left({{du, dv, duv/r}\atop{aduv, bduv}}; q, \frac{abr}{d}\right).
\end{align*}
\end{thm}
Setting $b=0$ in Theorem~\ref{ramliureci}, we immediately obtain Theorem~\ref{liuppreciprocity}.

Noting that when $r=duv,$ the $_3\phi_2$ series reduces to $1$. Thus, on putting $r=duv$ in Theorem~\ref{ramliureci},
we immediately obtain the following beautiful $q$-formula.
\begin{prop}\label{sramliureci} If $\rho$ is defined as in Definition~\ref{liudefn} with
$uv\not=0$ and
\[
\max\{|au|, |av|, |bu|, |bv|, |cu|, |cv|, |du|, |dv|, |abcdu^2v^2/q|\}<1,
\]
then, we have the following  six-variable reciprocity formula:
\begin{align*}
&\rho (a, b, c, d, duv, u, v)-\rho (a, b, c, d, duv, v, u)\\
&=\frac{\Delta(u, v)(abuv, acuv, aduv, bcuv, bduv, cduv; q)_\infty}
{(au, av, bu, bv, cu, cv, du, dv, abcdu^2v^2/q; q)_\infty}.
\end{align*}
\end{prop}
The remainder of this paper is organized as follows. Some inequalities for $q$-series are discussed in Section~2.
In Section~3, we introduce some important facts in $q$-differential calculus. Sections~4 and 5 are devoted to the proofs
of Theorems~\ref{liuqtrans} and \ref{ramliureci}. In Section~6, we will use  Theorems~\ref{liuqtrans} to
derive a beta integral formula  which including the Askey--Wilson integral as a special case.
Some limiting cases of Theorem~\ref{liuppreciprocity} are discussed in Section~7, and
one notable example is the following formula:
\[
\frac{(q; q)_\infty^4}{(qa, q/a; q)_\infty^2}
=1+(1-a)^2 \sum_{n=1}^\infty \frac{n(q/a)^n}{1-aq^n}
+(1-1/a)^2 \sum_{n=1}^\infty \frac{n (qa)^n}{1-q^n/a}.
\]
\section{some inequalities for $q$-series}
\noindent
For convenience, in this section,  Sections 3 and 4,   we assume  that $0<q<1.$

\begin{prop}\label{liuequality} If $k$ is a nonnegative integer or $\infty$, $a$ and $b$ are two nonnegative  numbers
such that $0 \le b \le 1, $ then, we have
\[
(-ab; q)_k \le (-a; q)_\infty.
\]
If we further assume that $0\le a \le 1$,  then, we have
\[
(ab; q)_k \ge (a; q)_\infty.
\]

\end{prop}
\begin{proof} Keeping the fact $0<q<1$ in mind, we find that for any $0\le j\le k-1$,
\[
1+abq^j \le 1+aq^j.
\]
On multiplying these inequalities together, we deduce that
\[
(-ab; q)_k \le (-a; q)_k.
\]
Since $(-aq^k; q)_\infty\ge 1,$ we multiply $(-aq^k; q)_\infty$ to the right-hand side of the above inequality
to arrive at the first inequality in the proposition. In the same way we can prove the second inequality.
This completes the proof of Proposition~\ref{liuequality}.
\end{proof}
\begin{prop}\label{convergenceseries} If $\max\{|b_1|, |b_2|, \ldots, |b_r|, |x|\}<1$ and $n$ is a nonnegative integer,
then, we have
\[
\left|{_{r+1}\phi_r}\left({{a, a_1q^n, \ldots, a_r q^n}\atop{b_1q^n, \ldots b_r q^n}}; q, x\right)\right|
\le \frac{(-|ax|, -|a_1|, \ldots, -|a_r|; q)_\infty}{(|x|, |b_1|, \ldots, |b_r|; q)_\infty}.
\]
\end{prop}
\begin{proof}Keeping $0<q<1$ in mind, using the triangle inequality and Proposition~\ref{liuequality},
we find that for $j\in\{1, 2, \ldots, r\}$,
\begin{equation*}
|(a_jq^n; q)_k|\le \prod_{l=0}^{k-1}(|1+|a_j|q^l)\le \prod_{l=0}^{\infty}(|1+|a_j|q^l)
=(-|a_j|; q)_\infty,
\end{equation*}
and
\begin{equation*}
|(b_jq^n; q)_k|\ge \prod_{l=0}^{k-1}(|1-|b_j|q^l)\ge \prod_{l=0}^{\infty}(|1-|b_j|q^l)
=(|b_j|; q)_\infty.
\end{equation*}
It follows that
\[
\left| \frac{(a, a_1q^n, \ldots, a_r q^n; q )_k x^k}{(q, b_1q^n, \ldots b_r q^n; q)_k} \right|\le
\frac{( -|a_1|, \ldots, -|a_r|)_\infty (-|a|; q)_k |x|^k}{(|b_1|, \ldots, |b_r|)_\infty (q; q)_k}.
\]
Using this inequality and  the triangle inequality, we conclude that
\begin{align*}
&\left|{_{r+1}\phi_r}\left({{a, a_1q^n, \ldots, a_r q^n}\atop{b_1q^n, \ldots b_r q^n}}; q, x\right)\right|\\
&\le \frac{( -|a_1|, \ldots, -|a_r|)_\infty }{(|b_1|, \ldots, |b_r|)_\infty }
\sum_{k=0}^\infty \frac{(-|a|; q)_k |x|^k}{(q; q)_k}.
\end{align*}
On applying the $q$-binomial theorem to the right-hand side of the above inequality, we complete
the proof of Proposition~\ref{convergenceseries}.
\end{proof}
It should be pointed out that Wang \cite[Theorem~1.1]{Wang} has obtained a similar inequality.
\section{some facts in $q$-differential calculus}
Next we introduce some basic concepts in $q$-differential calculus.
\begin{defn}\label{qderivative}
For any function $f(x)$ of one variable, the  $q$-derivative of $f(x)$
with respect to $x,$ is defined as
\begin{equation*}
\mathcal{D}_{q,x}\{f(x)\}=\frac{f(x)-f(qx)}{x},
\end{equation*}
and we further define  $\mathcal{D}_{q,x}^{0} \{f\}=f$~ and~ $\mathcal{D}_{q, x}^n \{f\}=\mathcal{D}_{q, x}\{\mathcal{D}_{q, x}^{n-1}\{f\}\}.$
\end{defn}
The $q$-derivative was first introduced by L.  Schendel \cite{Schnendel} in 1877 and then by
F. H. Jackson \cite{Jackson1908} in 1908, which is a $q$-analog of the ordinary derivative.
The definition of the $q$-partial derivative can be found in \cite{LiuRam2013}.
\begin{defn}\label{qpdfn}
A $q$-partial derivative of a function of several variables is its $q$-derivative with respect to one of those variables, regarding other variables as constants. The $q$-partial derivative of a function $f$ with respect to the variable $x$ is denoted by $\partial_{q, x}\{f\}$.
\end{defn}

\begin{defn}\label{qpde}
A $q$-partial differential equation is an equation that contains unknown multivariable functions and their $q$-partial derivatives.
\end{defn}
The  homogeneous Rogers--Szeg\H{o} polynomials play an important role in the theory of orthogonal
polynomials, which are defined by \cite{Liu2010, LiuRam2013}
\begin{equation}
h_n (a, b| q)=\sum_{k=0}^n  {n \choose k}_{q}a^kb^{n-k}.
\label{rseqn1}
\end{equation}
By multiplying two copies of the $q$-binomial theorem (see, for example \cite[p. 8, Eq. (1.3.2)]{Gas+Rah}),
one can find that \cite{Liu2010, LiuRam2013}
\begin{equation}
\sum_{n=0}^\infty h_n (a, b | q) \frac{t^n}{(q; q)_n}=\frac{1}{(at, bt; q)_\infty},  \quad \ |at|<1, |bt|<1.
\label{rseqn2}
\end{equation}

It turn out that the $q$-partial differential equations is an important subject of study,
we started the study of this subject in \cite{LiuRam2013} and \cite{LiuRam2014}.
The following very useful expansion theorem for $q$-series can be found in \cite[Proposition~1.6]{LiuRam2013}.

\begin{thm}\label{liuqexpansionthm} If $f(x,y)$  is a two-variable
 analytic function at $(0,0)\in \mathbb{C}^2$, then, $f$ can be expanded
 in terms of $h_n(x, y|q)$ if and only if $f$
 satisfies the $q$-partial differential equation
$
\partial_{q, x}\{f\}=\partial_{q, y}\{f\}.
$
\end{thm}
One of the most important formulae for the Rogers--Szeg\H{o} polynomials is the
following $q$-Mehler formula, which can be derived easily from Theorem~\ref{liuqexpansionthm},
see \cite[pp.219--220]{LiuRam2013} for details.
\begin{prop}\label{qmehler} For $\max\{|asz|, |atz|, |bsz|, |btz|\}<1,$ we have
\begin{align*}
\sum_{n=0}^\infty h_{n}(a, b|q)h_n(s, t|q) \frac{z^n}{(q; q)_n}
=\frac{(abstz^2; q)_\infty}{(asz, atz, bsz, btz; q)_\infty}.
\end{align*}
\end{prop}
In order to prove Theorems~\ref{liuqtrans} and \ref{ramliureci}, we need  the
following proposition.
\begin{prop}\label{liuqpde}
The function $L(a, b, u, v, s, t)$ satisfies the $q$-partial differential equation
$\partial_{q, a}\{L\}=\partial_{q, b}\{L\},$ where $L(a, b, u, v, s, t)$
is defined by
\[
\frac{(av, bv, abstu/v; q)_\infty}{(as, at, au, bs, bt, bu; q)_\infty}
{_3\phi_2}\left({{v/s, v/t, v/u}\atop{av, bv}}; q, \frac{abstu}{v}\right).
\]
\end{prop}
\begin{proof}  It is easily seen that  using $L(a, b, u, v, s, t)$ we can rewrite
the formula in Proposition~\ref{liuqint} in the form
\[
L(a, b, u, v, s, t)=\frac{(v/s, v/t; q)_\infty}{(1-q)\Delta(s, t) (au, bu; q)_\infty}
\int_{s}^t \frac{(qx/s, qx/t, abux; q)_\infty}{(ax, bx, vx/st;q)_\infty} d_q x.
\]
Noting the definition of the $q$-partial differential equations and using a direct computation, we easily find that
\begin{align*}
&\partial_{q, a}\{L\}=\partial_{q, b}\{L\}\\
&=\frac{(v/s, v/t; q)_\infty}{(1-q)\Delta(s, t)}
\int_{s}^t \frac{\left(x+u-aux-bux\right)(qx/s, qx/t, abuxq; q)_\infty}{(au, bu, ax, bx, vx/st;q)_\infty} d_q x,
\end{align*}
which indicates that Proposition~\ref{liuqpde} holds.
\end{proof}
\section{the proof of Theorem~\ref{liuqtrans}}
Recall the Sears $_3\phi_2$ transformation formula (see, for example \cite[Theorem~3]{Liu2003})
\begin{align*}
&{_3\phi_2}\left({{a_1, a_2, a_3}\atop{b_1, b_2}}; q,  \frac{b_1b_2}{a_1a_2a_3}\right)\\
&=\frac{(b_2/a_3, b_1b_2/a_1a_2; q)_\infty}{(b_2, b_1b_2/a_1a_2a_3; q)_\infty}
{_3\phi_2}\left({{b_1/a_1, b_1/a_2, a_3}\atop{b_1, b_1b_2/a_1a_2}}; q,  \frac{b_2}{a_3}\right).
\end{align*}
Using the Sears $_3\phi_2$ transformation formula, we easily conclude that
\begin{align*}
&{_3\phi_2}\left({{ax, ar, cx}\atop{acduvx, abrx}}q, bduv\right)\\
&=\frac{(abr/c, bcduvx; q)_\infty}{(bduv, abrx; q)_\infty}{_3\phi_2}\left({{cduv, cduvx/r, cx}\atop{acduvx, bcduvx}}; q, \frac{abr}{c}\right).
\end{align*}
This transformation formula shows  that Theorem~\ref{liuqtrans} is equivalent to the following formula:
\begin{align}
&\int_{u}^v \frac{(qx/u, qx/v, acduvx, bcduvx; q)_\infty}{(ax, bx, cx, dx; q)_\infty}\label{pliu:eqn1}\\
&\qquad \times {_3\phi_2}\left({{cduv, cduvx/r, cx}\atop{acduvx, bcduvx}}q, \frac{abr}{c}\right)d_q x\nonumber\\
&=\frac{(1-q)\Delta(u, v)(acuv, aduv, bcuv, bduv, cduv; q)_\infty}
{(au, av, bu, bv, cu, cv, du, dv; q)_\infty}\nonumber\\
&\qquad \times {_3\phi_2}\left({{cu, cv, cuv/r}\atop{acuv, bcuv}}; q, \frac{abr}{c}\right).\nonumber
\end{align}
In order to prove the identity (\ref{pliu:eqn1}), we need to prove the following lemma.
\begin{lem}\label{liuqintlemm}
The $q$-integral in (\ref{pliu:eqn1}) is a two-variable analytic function of $a$ and $b$, which is analytic at $(0, 0)\in \mathbb{C}^2.$
\end{lem}
\begin{proof}
For the sake of convenience, we define the compact  notation $A_n$ and $B_n$ by
\begin{align*}
A_n(a, b, c, d, r,  u, v):&={_3\phi_2}\left({{cduv, cdvu^2q^n/r, cuq^n}\atop{acdvu^2q^n, bcdvu^2q^n}}; q, \frac{abr}{c}\right),\\
B_n(a, b, c, d, r, u, v):&=u(1-q)\frac{(q^{n+1}, q^{n+1}u/v, acdvu^2 q^n, bcdvu^2 q^n; q )_\infty}{(auq^n, buq^n, cuq^n, duq^n; q)_\infty}.
\end{align*}
Using the definition of $q$-integral, we find that the left-hand side of the equation in (\ref{pliu:eqn1}) can
be written as
\begin{align}
&\sum_{n=0}^\infty A_n(a, b, c, d, r, v, u)B_n(a, b, c, d, r, v, u)q^n \label{pliu:eqn2}\\
&-\sum_{n=0}^\infty A_n(a, b, c, d, r, u, v)B_n(a, b, c, d, r, u, v)q^n.\nonumber
\end{align}

Next we will prove that this series converges to a two-variable analytic function of $a$ and $b$ at $(0, 0)\in \mathbb{C}^2.$
It is obvious that the first summation can be obtained from the second one by interchanging $u$ and $v$. Thus we only need consider
the second summation, namely,
\begin{equation}
\sum_{n=0}^\infty A_n(a, b, c, d, r, u, v)B_n(a, b, c, d, r, u, v)q^n.
\label{pliu:eqn3}
\end{equation}
We divide our proof into two cases according $cr\not=0$ and $cr=0$. We only prove the $cr\not=0$ case
and the $cr=0$ case can be proved in the same way. Without loss of generality, we can assume that
\[
\max\{|a|, |b|, |d|\}<1~\text{and}~0<|c|, |r|,  |u|, |v|<1.
\]
Using Propositions~\ref{liuequality} and \ref{convergenceseries} and doing some simple calculations, we find that
\begin{align*}
|A_n(a, b, c, d, r, u, v)|&\le \frac{(-|abduvr|, -|cdvu^2/r|, -|cu|; q)_\infty}{(|abr/c|, |acdvu^2|, |bcdvu^2|; q)_\infty}\\
&\le \frac{(-1, -|r|, -|1/r|; q)_\infty}{(|r/c|, |u|, |u|; q)_\infty}.
\end{align*}
On making use of  Proposition~\ref{liuequality} and some elementary calculations, we  deduce that
\[
|B_n(a, b, c, d, r, u, v)|\le \frac{(-1; q)_\infty^3 (-|qu/v|; q)_\infty}{(|u|; q)^4_\infty}.
\]
Using the triangular inequality and the above two inequalities, we conclude that
\begin{align*}
&|\sum_{n=0}^\infty A_n(a, b, c, d, r, u, v)B_n(a, b, c, d, r, u, v)q^n|\\
&\le \sum_{n=0}^\infty|A_n(a, b, c, d, r, u, v)||B_n(a, b, c, d, r, u, v)|q^n\\
&\le \frac{(-1; q)_\infty^4 (-|qu/v|, -|r|, -|1/r|; q)_\infty}{(1-q)(|u|; q)^6_\infty (|r/c|; q)_\infty}.
\end{align*}
This indicates that the series in (\ref{pliu:eqn3}) converges
absolutely and uniformly for $|a|<1.$ It is easily to see that every term of this series
is a analytic at $a=0$. Thus, this series converges to an analytic function of $a$, which
is  analytic at $a=0$.

By symmetry, this series also converges to an analytic function of $b$, which
a analytic at $b=0$. Hence the series in (\ref{pliu:eqn2}) converges to a two-variable analytic
function of $a$ and $b$, which is analytic at $(a, b)=(0, 0)\in \mathbb{C}^2$.

Interchanging $u$ and $v$ in (\ref{pliu:eqn3})  we immediately find that the first series in (\ref{pliu:eqn2})
is also analytic at $(a, b)=(0, 0)\in \mathbb{C}^2$.
In summary, the left-hand side of the equation in (\ref{pliu:eqn1}) is a two-variable analytic function
of $a$ and $b$, which is analytic at $(a, b)=(0, 0)\in \mathbb{C}^2$.
\end{proof}
Now we begin to prove (\ref{pliu:eqn1}) by using Lemma~\ref{liuqintlemm}, Proposition~\ref{liuqpde}
and Theorem~\ref{liuqexpansionthm}.
\begin{proof}
Using the definition of $L$ in Proposition~\ref{liuqpde} and some simple computations, we find that
\begin{align*}
L(a, b, duv, cduvx, x, r)&=\frac{(acduvx, bdcuvx, abr/c; q)_\infty}
{(ax, bx, ar, br, aduv, bduv; q)_\infty}\\
&\quad\times {_3\phi_2}\left({{cduv, cduvx/r, cx}\atop{acduvx, bcduvx}}; q, \frac{abr}{c}\right),
\end{align*}
\begin{align*}
L(a, b, r, cuv, u, v)&=\frac{(acuv, bcuv, abr/c; q)_\infty}{(au, av, bu, bv, ar, br; q)_\infty}\\
&\quad\times {_3\phi_2}\left({{cu, cv, cuv/r}\atop{acuv, bcuv}}; q, \frac{abr}{c}\right).
\end{align*}
Using these two equations we can rewrite (\ref{pliu:eqn1}) in the form
\begin{align}
&\int_{u}^v \frac{(qx/u, qx/v; q)_\infty}{(cx, dx; q)_\infty} L(a, b, duv, cduvx, x, r) d_q x \label{pliu:eqn4}\\
&\quad=\frac{(1-q)\Delta(u, v)(cduv; q)_\infty}{(cu, cv, du, dv; q)_\infty}L(a, b, r, cuv, u, v).\nonumber
\end{align}

If we use $f(a, b)$ to denote the left-hand side of (\ref{pliu:eqn4}), then, $f(a, b)$  is analytic at
$(0, 0)\in\mathbb{C}^2$, and satisfies the $q$-partial differential equation $\partial_{q, a}\{f\}=\partial_{q, b}\{f\}.$
Thus, by Theorem~\ref{liuqexpansionthm}, there exists a sequence $\{\alpha_n\}$ independent of $a$ and $b$ such that
\[
f(a, b)=\sum_{n=0}^\infty \alpha_n h_n(a, b|q).
\]
On putting $b=0$ in this equation,  using $h_n(a, 0|q)=a^n,$ and noting the definition of $f(a, b)$, we find that
\[
f(a, 0)=\frac{1}{(ar, aduv; q)_\infty}\int_{u}^v \frac{(qx/u, qx/v, acduvx; q)_\infty}{(ax, cx, dx; q)_\infty} d_q x=\sum_{n=0}^\infty \alpha_n a^n.
\]
Applying the Al--Salam--Verma integral to the $q$-integral in this equation, we deduce that
\begin{equation}
\sum_{n=0}^\infty \alpha_n a^n=\frac{(1-q)\Delta(u, v)(acuv, cduv; q)_\infty}{(au, av, ar, cu, cv, du, dv; q)_\infty}.
\label{pliu:eqn5}
\end{equation}

If we use $g(a, b)$ to denote the right-hand side of (\ref{pliu:eqn4}), then, $g(a, b)$  is analytic at
$(0, 0)\in\mathbb{C}^2$, and satisfies the $q$-partial differential equation $\partial_{q, a}\{g\}=\partial_{q, b}\{g\}.$
Thus, by Theorem~\ref{liuqexpansionthm}, there exists a sequence $\{\beta_n\}$ independent of $a$ and $b$ such that
\[
g(a, b)=\sum_{n=0}^\infty \beta_n h_n(a, b|q).
\]
On putting $b=0$ in this equation,  using $h_n(a, 0|q)=a^n,$ and noting the definition of $g(a, b)$, we find that
\[
g(a, 0)=\sum_{n=0}^\infty \beta_n a^n=\frac{(1-q)\Delta(u, v)(acuv, cduv; q)_\infty}{(au, av, ar, cu, cv, du, dv; q)_\infty}.
\]
Comparing this equation with (\ref{pliu:eqn5}), we find that $\alpha_n=\beta_n$, which implies that $f(a, b)=g(a, b)$,
which shows that the identity in (\ref{pliu:eqn1}) holds.
Thus we have proved Theorem~\ref{liuqtrans} for $|a|$ and $|b|$ sufficiently small.  By analytic continuation,
we complete the proof of  Theorem~\ref{liuqtrans}.
\end{proof}
\section{The proof of Theorem~\ref{ramliureci}}
In order to prove Theorem~\ref{ramliureci}, we first set up the following lemma.
\begin{lem}\label{liurecilem} The series on the left-hand side of the equation in Theorem~\ref{ramliureci}
represents a two-variable analytic function of $a$ and $b$, which is analytic at $(0, 0)\in \mathbb{C}^2.$
\end{lem}
\begin{proof}
The proof can be divided into two cases according to $rcd\not=0$ and $rcd=0.$
We only prove the $rcd\not=0$ case and the $rcd=0$ case can be proved similarly.

For the sake of simplicity, we will use the $C_n$ and $D_n$ to denote
\begin{align*}
C_n(a, b, c, d, u, v):&=v\frac{(q/du, acuv, bcuv; q)_n (dv)^n}{(av, bv, cv; q)_{n+1}},\\
D_n(a, b, c, r, u, v):&= {_3\phi_2}\left({{q^{n+1}, vq^{n+1}/r, q/cu}\atop{avq^{n+1}, bvq^{n+1}}}q; \frac{abcruv}{q}\right).
\end{align*}
Using these notations we can write the left-hand side of the equation in Theorem~\ref{ramliureci}
as
\begin{align}
&\sum_{n=0}^\infty C_n(a, b, c, d, u, v)D_n(a, b, c, r, u, v)\label{liureci:eqn1}\\
&-\sum_{n=0}^\infty C_n(a, b, c, d, v, u)D_n(a, b, c, r, v, u).\nonumber
\end{align}
Now we will show that this series converges to a two-variable analytic function of $a$ and
$b$ at $(0, 0)\in \mathbb{C}^2.$ It is obvious that the second summation in the above equation
can be obtained from the first summation by interchanging $u$ and $v$, so we only need consider
the first summation
\begin{align}
\sum_{n=0}^\infty C_n(a, b, c, d, u, v)D_n(a, b, c, r, u, v)\label{liureci:eqn2}
\end{align}

Without loss of generality, we may assume that $\max\{|a|, |b|\}<1$ and in order to simplify
the discussion, we only consider the case $rcd\not=0,$ as the case $rcd=0$ is similar. 
Using Proposition~\ref{liuequality} and some simple calculation, we conclude  that
\[
|C_n(a, b, c, d, u, v)|\le \frac{|v|(-|1/du|, -|v|, -|v|; q)_\infty|dv|^n}{(|v|, |v|, |cv|; q)_\infty}.
\]
On making use of  Proposition~\ref{convergenceseries} and a direct computation, we easily deduce that
\begin{align*}
|D_n(a, b, c, r, u, v)|&\le \frac {(-|abrv|, -q, -|v/r|; q)_\infty}
{(|v|, |v|, |q/cu|; q)_\infty}\\
&\le \frac {(-|rv|, -q, -|v/r|; q)_\infty}
{(|v|, |v|, |q/cu|; q)_\infty}.
\end{align*}

Using the triangular inequality and these two inequalities, we have
\begin{align*}
&\left|\sum_{n=0}^\infty C_n(a, b, c, d, u, v)D_n(a, b, c, r, u, v)\right|\\
&\le \sum_{n=0}^\infty |C_n(a, b, c, d, u, v)D_n(a, b, c, r, u, v)| \\
&\le |v|\frac{(-q, -|rv|, -|v/r|, -|1/du|; q)_\infty}
{(|v|; q)_\infty^4 (|cv|, |q/cu|; q)_\infty}
\sum_{n=0}^\infty |dv|^n\\
&=\frac{|v|(-q, -|rv|, -|v/r|, -|1/du|; q)_\infty}
{(1-|dv|)(|v|; q)_\infty^4 (|cv|, |q/cu|; q)_\infty}.
\end{align*}
This shows that the series in (\ref{liureci:eqn2}) converges absolutely and uniformly
for $\max\{|a|, |b|\}<1.$ It is easily seen that every term of this series is analytic
at $(a, b)=(0, 0)\in \mathbb{C}^2,$  thus this series converges to a two-variable analytic
function of $a$ and $b$ which is analytic at $(0, 0)\in \mathbb{C}^2.$
\end{proof}
Now we begin to prove Theorem~\ref{ramliureci}
by Theorem~\ref{liuqexpansionthm},  Proposition~\ref{liuqpde} and Lemma~\ref{liurecilem}.
\begin{proof} Using the definition of $L$ in Proposition~\ref{liuqpde}, we deduce that
\begin{align*}
&{_3\phi_2}\left({{du, dv, duv/r}\atop{aduv, bduv}}; q,  \frac{abr}{d}\right)\\
&=\frac {(ar, br, au, bu, av, bv; q)_\infty} {(aduv, bduv, abr/d; q)_\infty}
L(a, b, u, duv, v, r),
\end{align*}
\begin{align*}
&{_3\phi_2}\left({{q^{n+1}, vq^{n+1}/r, q/cu}\atop{avq^{n+1}, bvq^{n+1}}}; q, \frac{abcruv}{q}\right)\\
&=\frac{(av, bv, ar, br, acuvq^n, bcuvq^n; q)_\infty} {(avq^{n+1}, bvq^{n+1}, abcruv/q; q)_\infty}
L(a, b, r, vq^{n+1}, v, cuvq^n).
\end{align*}
Using these two equations we can rewrite Theorem~\ref{ramliureci} in the form
\begin{align}
&v\sum_{n=0}^\infty \frac{(q/du; q)_n (dv)^n}{(cv; q)_{n+1}} L(a, b, r, vq^{n+1}, v, cuvq^n)\label{liureci:eqn3}\\
&-u\sum_{n=0}^\infty \frac{(q/dv; q)_n (du)^n}{(cv; q)_{n+1}} L(a, b, r, uq^{n+1}, u, cuvq^n)\nonumber\\
&=\frac{\Delta(u, v)(cduv; q)_\infty}{(cu, cv, du, dv; q)_\infty} L(a, b, u, duv, v, r).\nonumber
\end{align}

If we use $f(a, b)$ to denote the left-hand side of this equation, then, from Lemma~\ref{liurecilem}
we know that $f(a, b)$ is analytic at $(0, 0)\in \mathbb{C}^2.$ Using Proposition~\ref{liuqpde}, we
easily see that $f(a, b)$ satisfies the $q$-partial differential equation
$\partial_{q, a}\{f\}=\partial_{q, b}\{f\}.$  Thus by Theorem~\ref{liuqexpansionthm}, there exists
a sequence $\{\alpha_n\}$ independent of $a$ and $b$ such that
\[
f(a, b)=\sum_{n=0}^\infty \alpha_n h_n(a, b).
\]
Putting $b=0$ in this equation and using the fact $h_n(a, 0)=a^n$,
we obtain
\[
f(a, 0)=\sum_{n=0}^\infty \alpha_n a^n.
\]
Noting the definition of $f(a, b)$ and using Theorem~\ref{liuppreciprocity},  we find that
\begin{align*}
&(ar, acuv; q)_\infty f(a, 0)\\
&= v\sum_{n=0}^\infty \frac{(q/du, acuv; q)_n(dv)^n}{(av, cv; q)_{n+1}}
-u\sum_{n=0}^\infty \frac{(q/dv, acuv; q)_{n}(du)^n}{(au, cu; q)_{n+1}}\\
&=\frac{\Delta(u, v)( aduv, acuv, cduv; q)_\infty}{(au, av, cu, cv, du, dv; q)_\infty}.
\end{align*}

If we use $g(a, b)$ to denote the right-hand side of (\ref{liureci:eqn3}), then, $g(a, b)$  is analytic at
$(0, 0)\in\mathbb{C}^2$, and satisfies the $q$-partial differential equation $\partial_{q, a}\{g\}=\partial_{q, b}\{g\}.$
Thus, by Theorem~\ref{liuqexpansionthm}, there exists a sequence $\{\beta_n\}$ independent of $a$ and $b$ such that
\[
g(a, b)=\sum_{n=0}^\infty \beta_n h_n(a, b|q).
\]
On putting $b=0$ in this equation,  using $h_n(a, 0|q)=a^n,$ and noting the definition of $g(a, b)$, we find that
\[
g(a, 0)=\sum_{n=0}^\infty \beta_n a^n=\frac{\Delta(u, v) (aduv, cduv; q)_\infty}{(ar, au, av, cu, cv, du, dv; q)_\infty}.
\]
It follows that
\[
\sum_{n=0}^\infty \alpha_n a^n=\sum_{n=0}^\infty \beta_n a^n.
\]
Thus we have $\alpha_n=\beta_n,$ which implies that $f(a, b)=g(a, b).$ Hence we have
proved Theorem~\ref{ramliureci} for $|a|$ and $|b|$ sufficiently small and $0<q<1$. Using analytic
continuation, this completes the proof of Theorem~\ref{ramliureci}.
\end{proof}

\section{A beta integral formula}
\begin{defn}\label{defhermite}
For$\ x=\cos \T$, we define  $h(x; a)$ and
 $ h(x; a_1, a_2, \ldots, a_m)$  as follows:
 \begin{align*}
&h(x; a)=(ae^{i\T}, ae^{-i\T}; q)_\infty=\prod_{k=0}^\infty (1-2q^k ax+q^{2k}a^2), \\
&h(x; a_1, a_2, \ldots, a_m)=h(x; a_1)h(x; a_2)\cdots h(x; a_m).
 \end{align*}
\end{defn}
In this section we will use Theorem~\ref{liuqtrans} to prove the following beta integral formula,
which including  the Askey--Wilson integral as a special case.
\begin{thm} \label{betaint}For $\max\{|a|, |b|, |c, |d|, |abr/c|\}<1,$ we have the integral formula
\begin{align*}
&\frac{2\pi a_0}{(q, ac, ad, bc, cd, abr/c; q)_\infty}=\frac{2\pi (abdr; q)_\infty}{(q, ac,ad,  bc, bd, cd, abr/c; q)_\infty}\\
&=\int_{0}^{\pi}
\frac{h(\cos 2\theta; 1)}
{ h(\cos \theta; a, b, c, d)} {_3\phi_2}\left({{ce^{i\theta}, ce^{-i\theta}, c/r}\atop{ac, bc}}; q, \frac{abr}{c}\right) d\theta.
\end{align*}
\end{thm}
\begin{proof}
For the sake of brevity, we use $I(x)$ to denote the function
\begin{equation}
I(x)=\frac{(acdx, abrx; q)_\infty}{(ax, bx, cx, dx; q)_\infty}
{_3\phi_2}\left({{ar, ax, cx}\atop{acdx, abrx}}q, bd\right).
\label{awl:eqn1}
\end{equation}
It is easily seen that $I(x)$ is analytic near $x=0$.
Thus, there exists a sequence $\{a_k\}_{k=0}^\infty$ independent of $x$ such
that
\[
I(x)=a_0+\sum_{k=1}^\infty a_k x^k.
\]
By setting $x=0$ in the above equation and using the $q$-binomial theorem, we  conclude that
\begin{equation}
a_0=\sum_{n=0}^\infty \frac{(ar; q)_n (bd)^n}{(q; q)_n}=\frac{(abdr; q)_\infty}{(bd; q)_\infty}.\label{awl:eqn2}
\end{equation}
Noting the definition of $\Delta(u, v)$ in (\ref{theta:eqn1}) and using a simple computation, we easily find that
\begin{equation}
(e^{i\theta}-e^{-i\theta})\Delta(e^{i\theta}, e^{-i\theta})=(q; q)_\infty h(\cos 2\theta; 1).
\label{awl:eqn3}
\end{equation}

On  replacing $(u, v)$ by $(e^{i\theta}, e^{-i\theta})$ in Theorem~\ref{liuqtrans} and noting
the above equation , we deduce that
\begin{align}
&(e^{i\theta}-e^{-i\theta})\int_{e^{i\theta}}^{e^{-i\theta}} (qxe^{i\theta}, qxe^{-i\theta}; q)_\infty I(x)d_q x
\label{awl:eqn4}\\
&=\frac{(1-q)(q, ac, ad,  bc, cd, abr/c; q)_\infty h(\cos 2\theta; 1)}{h(\cos \theta; a, b, c, d)}\nonumber\\
&\qquad \times {_3\phi_2}\left({{ce^{i\theta}, ce^{-i\theta}, c/r}\atop{ac, bc}}; q, \frac{abr}{c}\right).\nonumber
\end{align}
Using the definition of the $q$-integral,  we find that the left-hand side
of this equation equals
\begin{align*}
&(1-q)(1-e^{-2i\theta}) \sum_{n=0}^\infty (q^{n+1}; q)_\infty (q^n e^{-2i\theta}; q)_\infty I(q^n e^{-i\theta})q^n\\
&+(1-q)(1-e^{2i\theta})\sum_{n=0}^\infty (q^{n+1}; q)_\infty (q^n e^{2i\theta}; q)_\infty I(q^n e^{i\theta})q^n.
\end{align*}
Inspecting the first series in the equation,  we see that
this series can be expanded in terms of the negative powers of
$\{e^{-ki\T}\}_{k=0}^\infty$, and the constant term of the
Fourier expansion of this series is $(1-q)a_0$. Thus, there exists a
sequence $\{\alpha_k\}_{k=1}^\infty$ independent of $\theta$ such that the first series
equals
\[
(1-q)a_0+\sum_{k=1}^\infty\alpha_k e^{-ik\theta}.
\]
Replacing $\theta$ by $-\theta$, we immediately find that the second series is equal to
\[
(1-q)a_0+\sum_{k=1}^\infty \alpha_k e^{ik\theta}.
\]
Combining the above two expressions together, we arrive at
\begin{align*}
&(e^{i\theta}-e^{-i\theta})\int_{e^{i\theta}}^{e^{-i\theta}} (qxe^{i\theta}, qxe^{-i\theta}; q)_\infty
I(x)d_q x\\
&\qquad=2(1-q)a_0+2\sum_{k=1}^\infty \alpha_k \cos k\theta.
\end{align*}
Comparing this equation with (\ref{awl:eqn3}), we are led to the Fourier series expansion
\begin{equation*}
2(1-q)a_0+2\sum_{k=1}^\infty \alpha_k \cos k\theta
\end{equation*}
\begin{align*}
&=\frac{(1-q)(q,  ac, ad, bc, cd, abr/c; q)_\infty h(\cos 2\theta; 1)}{h(\cos \theta; a, b, c, d)}\\
&\quad \times {_3\phi_2}\left({{ce^{i\theta}, ce^{-i\theta}, c/r}\atop{ac, bc}}; q, \frac{abr}{c}\right).
\end{align*}
On integrating  the above equation over $[-\pi, \pi]$ and using the fact
\[
\int_{-\pi}^\pi (\cos k \theta) d\theta=2\pi \delta_{k, 0},
\]
and noting that the integrand is an even function of $\theta,$ we deduce that
\begin{align*}
&\int_{0}^{\pi}
\frac{h(\cos 2\theta; 1)}
{ h(\cos \theta; a, b, c, d)} {_3\phi_2}\left({{ce^{i\theta}, ce^{-i\theta}, c/r}\atop{ac, bc}}; q, \frac{abr}{c}\right) d\theta\\
&=\frac{2\pi a_0}{(q, ac, ad, bc, cd, abr/c; q)_\infty}.
\end{align*}
Substituting the value of $a_0$ in (\ref{awl:eqn2}) into this equation, we complete the proof of Theorem~\ref{betaint}.
\end{proof}
Letting $r=c$ in Theorem~\ref{betaint}, we immediately obtain the Askey--Wilson integral formula \cite{Ask+Wil}.
\begin{thm}  \label{askwilint} If  max$\{|a|, |b|, |c|,
|d|\}<1$, then, we have
\begin{equation*}
\int_{0}^{\pi} \frac{h(\cos 2\T; 1)}{h(\cos \T; a, b, c, d)}d\T
=\frac{2\pi (abcd; q)_\infty}{(q, ab, ac, ad, bc, bd, cd;
q)_\infty}.
\end{equation*}
\end{thm}
\begin{rem} \rm By setting $u=e^{i\theta}$ and $v=e^{-i\theta}$ in Proposition~\ref{sramliureci} and then using
the same argument as in the proof of  Theorem~\ref{betaint}, we can give a derivation of the Askey--Wilson
integral.
\end{rem}

The continuous $q$-Hermite polynomials $H_n(\cos \theta|q)$ is defined as
\begin{equation}
H_n(\cos \theta|q)=\sum_{k=0}^n {n \choose k}_q e^{i(n-2k)\theta}.
\label{qhermite;eqn1}
\end{equation}
Using the definition of  the homogeneous Rogers--Szeg\H{o} polynomials defined in (\ref{rseqn1}),
 it is easily seen that
 \begin{equation}
 H_n(\cos \theta|q)=h_n(e^{-i\theta}, e^{i\theta}|q).
 \label{qhermite;eqn2}
 \end{equation}
Putting $a=e^{-i\theta}$ and $b=e^{i\theta}$ in (\ref{rseqn2}), one can find
the following proposition.
\begin{prop}\label{Hermitepp} For $|t|<1,$ we have
\begin{equation}
\sum_{n=0}^\infty H_n(\cos \theta|q) \frac{t^n}{(q; q)_n}=\frac{1}{(te^{i\theta}, te^{-i\theta}; q)_\infty}.
\label{hermitegf}
\end{equation}
\end{prop}
\begin{rem} \rm
One of the most important properties of the $q$-Hermite polynomials is that they satisfy the
following orthogonality relation, which was first proved by Szeg\H{o} \cite{Szeg}:
\begin{equation*}
\int_{0}^\pi H_m(\cos \theta|q) H_n(\cos \theta|q) h(\cos 2\theta; 1)d\theta=2\pi (q; q)_n \delta_{m, n}/(q; q)_\infty.
\label{qhermite:eqn1}
\end{equation*}
This orthogonality relation has been used by several authors to evaluate the Askey--Wilson integral
and other related $q$-beta integrals (see, for example \cite{SalamIsmail, Ism+Sta, Liu1997}).
We have just evaluated the Askey--Wilson integral without using the orthogonality relation for the $q$-Hermite polynomials.
We can use a special case of the Askey--Wilson integral formula  to give a new proof of the the orthogonality relation for the $q$-Hermite polynomials.
 The proof is as follows:

 Putting $c=d=0$ in the Askey--Wilson integral formula, we immediately deduce that
\begin{equation}
\int_{0}^{\pi} \frac{h(\cos 2\T; 1)}{h(\cos \T; a, b)}d\T
=\frac{2\pi }{(q, ab; q)_\infty}.
\label{qhermite:eqn2}
\end{equation}

On multiplying two copies of the identity in (\ref{hermitegf}), we find that
for $\max\{|a|, |b|\}<1,$
\[
\sum_{m, n=0}^\infty H_m(\cos \theta| q) H_n(\cos \theta|q)\frac{a^m b^n}{(q; q)_m(q; q)_n}
=\frac{1}{h(\cos \theta; a, b)}.
\]
Using Proposition~\ref{Hermitepp}, we can easily show that the above double series converges uniformly
for  $\max\{|a|, |b|\}<1$ on $0\le \theta \le \pi$.

Substituting this series into the left-hand side of
(\ref{qhermite:eqn2}) and then integrating term by term, and applying the $q$-binomial theorem to the
right-hand side of (\ref{qhermite:eqn2}),  we find that
\begin{align*}
&\sum_{m, n=0}^\infty \frac{a^m b^n}{(q; q)_m(q; q)_n}\int_{0}^\pi H_m(\cos \theta|q) H_n(\cos \theta|q) h(\cos 2\theta; 1) d\theta\\
&=\frac{2\pi }{(q; q)_\infty}\sum_{n=0}^\infty \frac{(ab)^n}{(q; q)_n}.
\end{align*}
Equating the coefficients of $a^m b^n$ on both sides of this equation, we arrive at the orthogonality relation for the
$q$-Hermite polynomials.
\end{rem}
\section{Some limiting cases of Theorem~\ref{liuppreciprocity}}
In this section we will discussed some limiting cases of Theorem~\ref{liuppreciprocity}.
\begin{prop}\label{alimitcases} For $a\not=q^{-m}, c\not=q^{-m}, d\not=q^{-m}, m=0, 1, 2, \ldots,$
we have
\begin{align*}
&\sum_{n=0}^\infty \frac{(q/d, ac; q)_n d^n}{(a, c; q)_{n+1}}\\
&\qquad \times \left(n+1+\sum_{k=0}^n \frac{aq^k}{1-aq^k}+\sum_{k=0}^n \frac{cq^k}{1-cq^k}
-\sum_{k=1}^n \frac{q^k}{d-q^k}\right)\\
&=\frac{(q; q)_\infty^3 (ac, ad, cd; q)_\infty}{(a, c, d; q)_\infty^2}.
\end{align*}
\end{prop}
\begin{proof} Keeping in mind that $\Delta(u, v)=(v-u)(q, qu/v, qv/u; q)_\infty$, dividing both
sides of the equation in Theorem~\ref{liuppreciprocity} by $v-u$, then letting $v\to u$ in the resulting
equation, using L$'$H$\hat{\rm o}$spital's rule and simplifying, we deduce that
\begin{align*}
&\sum_{n=0}^\infty \frac{(q/du, acu^2; q)_n (du)^n}{(au, cu; q)_{n+1}}\\
&\qquad \times \left(n+1+\sum_{k=0}^n \frac{auq^k}{1-auq^k}+\sum_{k=0}^n \frac{cuq^k}{1-cuq^k}
-\sum_{k=1}^n \frac{q^k}{du-q^k}\right)\\
&=\frac{(q; q)_\infty^3 (acu^2, adu^2, cdu^2; q)_\infty}{(au, cu, du; q)_\infty^2}.
\end{align*}
On replacing $(au, bu, cu)$ by $(a, b, c)$, we complete the proof of  Proposition~\ref{alimitcases}.
\end{proof}
On setting $a=c=0$ in proposition~\ref{alimitcases}, we immediately obtain the following proposition.
\begin{prop}\label{limiteuler} For $d\not=q^{-m}, m=0, 1, 2, \ldots,$
we have
\begin{align*}
\sum_{n=0}^\infty (q/d; q)_n d^n \left(n+1-\sum_{k=1}^n \frac{q^k}{d-q^k}\right)
=\frac{(q; q)_\infty^3 }{(d; q)_\infty^2}.
\end{align*}
\end{prop}
On putting $d=0$ in Proposition~\ref{limiteuler}, we immediately obtain the following identity of Jacobi
\cite[p. 14]{Berndt2006}:
\[
(q; q)_\infty^3=\sum_{n=0}^\infty (-1)^n (2n+1) q^{n(n+1)/2}.
\]

On letting $d\to q$ Proposition~\ref{limiteuler}, we obtain the
Euler identity (see, for example \cite[p. 280]{Andrews1983})
\[
1-\sum_{n=1}^\infty (q; q)_{n-1} q^n=(q; q)_\infty.
\]

On taking $d=-q$ in Proposition~\ref{limiteuler}, we are led to the identity
\[
1+\sum_{n=1}^\infty (-q)^n (-q; q)_{n-1} \left(2n+3+2\sum_{k=1}^{n-1} \frac{q^k}{1+q^k}\right)
=\frac{(q; q)_\infty^3}{(-q; q)_\infty^2}.
\]
\begin{prop}\label{blimitcases} For $a\not=q^{m}, m=\pm 1, \pm 2, \ldots,$ we have
the Lambert series formula
\[
\frac{(q; q)_\infty^4}{(qa, q/a; q)_\infty^2}
=1+(1-a)^2 \sum_{n=1}^\infty \frac{n(q/a)^n}{1-aq^n}
+(1-1/a)^2 \sum_{n=1}^\infty \frac{n (qa)^n}{1-q^n/a}.
\]
\end{prop}
\begin{proof}
Setting $c=0$ and $ad=q$ in proposition~\ref{alimitcases}, we conclude that
\begin{align}
\frac{(q; q)_\infty^4}{(a, q/a; q)_\infty^2}
&=\sum_{n=0}^\infty \frac{(q/a)^n}{1-aq^n}\left(n+\frac{1}{1-aq^n}\right)\label{lim:eqn1}\\
&=\frac{1}{(1-a)^2}+\sum_{n=1}^\infty \frac{n(q/a)^n}{1-aq^n}
+\sum_{n=1}^\infty \frac{(q/a)^n}{(1-aq^n)^2}.\nonumber
\end{align}
By a direct calculation, we can find the following elementary identity:
\[
\sum_{n=1}^\infty \frac{(q/a)^n}{(1-aq^n)^2}
=a^{-2}\sum_{n=1}^\infty \frac{n (qa)^n}{1-q^n/a}.
\]
Substituting this equation into (\ref{lim:eqn1}) and then
multiplying both sides of the resulting equation by $(1-a)^2,$ we complete the proof
of Proposition~\ref{blimitcases}.
\end{proof}
On taking $a=-1$ Proposition~\ref{blimitcases}, we immediately conclude that
\[
\frac{(q; q)_\infty^4}{(-q; q)_\infty^4}
=1+8\sum_{n=1}^\infty \frac{n(-q)^n}{1+q^n}.
\]
On replacing $q$ by $-q$ in this equation, we can obtain Jacobi's four-square identity
(see, for example \cite[p. 61]{Berndt2006})
\[
\left(\sum_{n=-\infty}^\infty q^{n^2}\right)^4
=1+8\sum_{n=1}^\infty \frac{nq^n}{1-q^n}
-32\sum_{n=1}^\infty \frac{nq^{4n}}{1-q^{4n}}.
\]
On writing $q$ by $q^2$ and then setting $a=q$ in the first identity in (\ref{lim:eqn1}), we deduce that
\begin{align*}
\frac{(q^2; q^2)_\infty^4}{(q; q^2)_\infty^4}
&=\sum_{n=0}^\infty \frac{nq^n}{1-q^{2n+1}}
+\sum_{n=0}^\infty \frac{q^n}{(1-q^{2n+1})^2}\\
&=\sum_{n=0}^\infty \frac{nq^n}{1-q^{2n+1}}+\sum_{n=0}^\infty \frac{(n+1)q^n}{1-q^{2n+1}}\\
&=\sum_{n=0}^\infty \frac{(2n+1)q^n}{1-q^{2n+1}},
\end{align*}
which is equivalent to the Legendre four triangular numbers identity
(see, for example \cite[p. 72]{Berndt2006})
\[
\left( \sum_{n=0}^\infty q^{n(n+1)/2}\right)^4
=\sum_{n=0}^\infty \frac{(2n+1)q^n}{1-q^{2n+1}}.
\]

\end{document}